\documentclass[11pt]{article}
\usepackage{amssymb}
\usepackage{amsmath}
\usepackage{amsthm}
\usepackage{latexsym}
\usepackage{hyperref}
\usepackage{mathdots}
\usepackage[all]{xy}
\usepackage{stmaryrd}
\usepackage[normalem]{ulem}
\usepackage[capitalise]{cleveref}
\usepackage{tikz}
\usetikzlibrary{patterns}

\theoremstyle{definition}
\newtheorem{theo}{Theorem}[section]
\newtheorem{defi}[theo]{Definition}
\newtheorem{prop}[theo]{Proposition}

\newtheorem{exa}[theo]{Example}
\newtheorem{rem}[theo]{Remark}
\newtheorem{nota}[theo]{Notation}
\newtheorem{OpenP}[theo]{Open Problem}
\numberwithin{equation}{section}

\newcommand{\N}{{\mathbb N}}
\newcommand{\F}{{\mathbb F}}

\newcommand{\cC}{{\mathcal C}}
\newcommand{\cD}{{\mathcal D}}

\newcommand{\cF}{{\mathcal F}}

\newcommand{\cI}{{\mathcal I}}

\newcommand{\cM}{{\mathcal M}}

\newcommand{\cO}{{\mathcal O}}

\newcommand{\cL}{{\mathcal L}}

\newcommand{\cS}{{\mathcal S}}
\newcommand{\cN}{{\mathcal N}}
\newcommand{\cU}{{\mathcal U}}
\newcommand{\cV}{{\mathcal V}}

\newcommand{\cX}{{\mathcal X}}

\newcommand{\cZ}{{\mathcal Z}}

\renewcommand{\cdotp}{\raisebox{-.17ex}{$\boldsymbol{\cdot}$}}
\newcommand{\qM}{$q$-matroid}
\newcommand{\T}{\mbox{$^{\sf T}$}}
\newcommand{\subspace}[1]{\mbox{$\langle{#1}\rangle$}}

\newcommand{\rk}{{\rm rk}\,}

\newcommand{\GL}{{\rm GL}}
\newcommand{\rowsp}{\mbox{\rm rs}}

\newcommand{\supp}{\mbox{\rm supp}}
\newcommand{\wtrk}{\mbox{$\mbox{\rm wt}_{\mbox{\rm\footnotesize rk}}$}}

\newcommand{\Smalltwomat}[2]{\mbox{$\left(\begin{smallmatrix}{#1}\\[.4ex]{#2}\end{smallmatrix}\right)$}}
\newcommand{\Smallfourmat}[4]{\mbox{$\left(\begin{smallmatrix}{#1}&{#2}\\[.4ex]{#3}&{#4}\end{smallmatrix}\right)$}}
\newcounter{alp}
\newcounter{ara}
\newcounter{rom}

\newenvironment{alphalist}{\begin{list}{(\alph{alp})\hfill}{\usecounter{alp}
     \topsep-1.4ex \labelwidth.7cm \leftmargin.7cm \labelsep0cm
     \rightmargin0cm \parsep0ex \itemsep.5ex
     \partopsep-1.4ex}}{\end{list}}

\newenvironment{mylist2}{\begin{list}{(\arabic{ara})\hfill}{\usecounter{ara}
     \topsep-1ex \labelwidth.88cm \leftmargin.88cm \labelsep0cm
     \rightmargin0cm \parsep0ex \itemsep.5ex
     \partopsep-1ex}}{\end{list}}

\makeatletter
\let\@fnsymbol\@arabic
\makeatother
\begin{document}

\title{Representability of the Direct Sum of $q$-Matroids}
\author{Heide Gluesing-Luerssen\thanks{Department of Mathematics, University of Kentucky, Lexington KY 40506-0027, USA; heide.gl@uky.edu and benjamin.jany@uky.edu. HGL was partially supported by the grant \#422479 from the Simons Foundation.}\quad and Benjamin Jany\footnotemark[1]
}

\date{February 22, 2023}
\maketitle
	
\begin{abstract}
\noindent 
While there are many parallels between matroid theory and $q$-matroid theory, most notably on the level of cryptomorphisms,
there are substantial differences when it comes to the direct sum.
The direct sum of $q$-matroids has been introduced in the literature only recently.
In this short note we show that the direct sum of representable $q$-matroids may not be representable.
It remains an open question whether representability of the direct sum can be characterized by the given $q$-matroids.
\end{abstract}

\textbf{Keywords:} $q$-matroid, direct sum, representability.

\section{Introduction}\label{S-Intro}

A $q$-matroid is the $q$-analogue of a matroid.
The ground space is a  finite-dimensional vector space over some finite field and the rank function is defined on the subspace lattice of that ground space. It satisfies the natural analogues of the rank function of a matroid.
$q$-Matroids have garnered a lot of attention in the coding-theory community due to their close relation to rank-metric codes; we refer to
\cite{JuPe18,Shi19,GJLR19,GhJo20,BCIJ21,GLJ22Gen,JPV22} for further details.

The notion of a $q$-matroid has only been introduced in 2018 in \cite{JuPe18}, and its theory is still in the beginnings.
Fortunately, a long catalogue of cryptomorphisms has already been established in \cite{BCJ22}.
It tells us that $q$-matroids can be defined by any of its abundance of structures, such as independent spaces, bases, circuits, flats, etc, and
provides us with  $q$-analogues of the axioms that need to be satisfied.
More recently, an additional cryptomorphism, based on cyclic flats, has been derived in~\cite{AlBy22}.

Despite this extensive list of cryptomorphisms, it turns out that the notion of direct sum for $q$-matroids is far from obvious.
None of the many ways of defining the direct sum of classical matroids has a well-defined $q$-analogue (see \cref{R-DirSumClassMatr} for further details).
Furthermore, in \cite{GLJ21C} it has been shown that there is no coproduct in the category of $q$-matroids with strong maps as morphisms.
This stands in contrast to the matroid case and suggests that a direct sum of $q$-matroids will not behave well with respect to its flats.
Nonetheless, just recently a notion of direct sum of $q$-matroids has been introduced in~\cite{CeJu21} by way of the rank function.
In the same paper some essential properties of the direct sum have been established, and in \cite{GLJ21C}
it has been shown that it is a coproduct in the category of $q$-matroids with linear weak maps as morphisms.
The latter implies that the direct sum of $q$-matroids on ground spaces~$E_1$ and~$E_2$  is the unique $q$-matroid
with the most independent spaces among all $q$-matroids on $E_1\oplus E_2$  whose restrictions to~$E_i$ are isomorphic to the given
$q$-matroids.
All these results suggest that the proposed notion of direct sum is the appropriate one, and it remains to investigate its properties.

In this short note we will discuss representability of the direct sum.
For a $q$-matroid whose ground space is an $\F_q$-vector space, representability can be naturally defined over any field extension
of~$\F_q$; see \cref{T-ReprqMatr} below.
In contrast to the representation of classical matroids, this fixes a priori the characteristic of the `representing field', and the only free parameter is
the degree of the field extension.
So far, very little is known about representability of \qM{}s (such as any obstruction criteria).

In this short paper we show the somewhat surprising fact that the direct sum of two $\F_{q^m}$-represent\-able $q$-matroids may
not be $\F_{q^m}$-representable or, even worse, not be representable over any field extension of~$\F_q$.
In order to present examples we will derive the following results. Firstly, if the direct sum is representable then it is so by a block diagonal matrix whose blocks represent the summands (\cref{T-DSReprMat}).
Secondly, we introduce the notion of paving $q$-matroids as the natural $q$-analogue of paving matroids.
For representable paving $q$-matroids of the same rank we derive a condition on the representing matrices
that prevents the block diagonal matrix to represent the direct sum (\cref{T-SuppCPerp}).
This will allow us to construct the desired examples.
It remains an open question whether representability of the direct sum can in general be characterized in terms of the summands.

\textbf{Notation:} Throughout, let $\F=\F_q$.
For any finite-dimensional $\F$-vector space~$E$ we denote by $\cL(E)$  the subspace lattice of~$E$.
For a matrix $M\in\F_{q^a}^{k\times n}$ we use $\rowsp(M)\subseteq\F_{q^a}^n$ for the row space of~$M$.
The abbreviation RREF stands for reduced row echelon form.
We set $\F^{\cdotp\times n}:=\bigcup_{y=1}^n\F^{y\times n}$, that is, $\F^{\cdotp\times n}$ is the set of all matrices over~$\F$ with $n$ columns and at most~$n$ rows. The collection of all their row spaces is $\cL(\F^n)$.

\section{The Direct Sum of $q$-Matroids}

\begin{defi}\label{D-qMatroid}
A \emph{$q$-matroid with ground space~$E$} is a pair $\cM=(E,\rho)$, where~$E$ is a finite-dimensional $\F$-vector space
and $\rho: \cL(E)\longrightarrow\N_{\geq0}$ is a map satisfying
\begin{mylist2}
\item[(R1)\hfill] Dimension-Boundedness: $0\leq\rho(V)\leq \dim V$  for all $V\in\cL(E)$;
\item[(R2)\hfill] Monotonicity: $V\leq W\Longrightarrow \rho(V)\leq \rho(W)$  for all $V,W\in\cL(E)$;
\item[(R3)\hfill] Submodularity: $\rho(V+W)+\rho(V\cap W)\leq \rho(V)+\rho(W)$ for all $V,W\in\cL(E)$.
\end{mylist2}
The value $\rho(V)$ is called the \emph{rank of}~$V$ and $\rho(\cM):=\rho(E)$ the \emph{rank} of the \qM.
A subspace $V\in\cL(E)$ is \emph{independent} if $\rho(V)=\dim V$ and \emph{dependent} otherwise.
The dependent spaces all of whose proper subspaces are independent are called \emph{circuits}.
A space $V\in\cL(E)$ is a \emph{flat} if it is inclusion-maximal in the set $\{W\in\cL(E)\mid \rho(W)=\rho(V)\}$, and a space is
\emph{open} if it is the sum of circuits.
\end{defi}

\begin{rem}\label{R-Crypto}
It is well known (see \cite[Thm.~8]{JuPe18})  that the collection of independent spaces uniquely determines the $q$-matroid, and the same is
true for the collections of dependent spaces, open spaces, and circuits.
For this and many more cryptomorphisms for $q$-matroids we refer to \cite{BCJ22}.
Furthermore, together with their rank values the cyclic flats also uniquely determine the $q$-matroid; see \cite{AlBy22,GLJ23DSCyc} and specifically~\cite{AlBy22} for the resulting cryptomorphism.
\end{rem}

An abundance of $q$-matroids is obtained as follows. Recall that $\rowsp(\cdot)$ denotes the row space of a matrix.

\begin{theo}[\mbox{\cite[Sec.~5]{JuPe18}}]\label{T-ReprqMatr}
Let $\F_{q^m}$ be a field extension of $\F=\F_q$ and let $G\in\F_{q^m}^{k\times n}$.
Define the map $\rho:\cL(\F^n)\longrightarrow\N_0$ via
\[
    \rho(V)=\rk(G Y\T),\ \text{ where $Y\in\F^{\cdotp\times n}$ is such that $V=\rowsp(Y)$}
\]
(and where the rank is taken over $\F_{q^m}$).
Then~$\rho$ is well-defined and $\cM_G:=(\F^n,\rho)$ is a $q$-matroid. It is called the $q$-matroid \emph{represented by $G$}.
A $q$-matroid $\cM$ with ground space~$\F^n$ is called \emph{$\F_{q^m}$-representable} if $\cM=\cM_G$ for some matrix~$G\in\F_{q^m}^{\:\cdotp\times n}$ (whose rank equals the rank of~$\cM$).
The $q$-matroid $\cM$ is called \emph{representable} if it is $\F_{q^m}$-representable over some field extension~$\F_{q^m}$.
\end{theo}

In short, for a representable \qM{} $\cM_G=(\F^n,\rho)$ with $G\in\F_{q^m}^{\:\cdotp\times n}$  we have
\begin{equation}\label{e-rkGY}
     \rho(\rowsp(Y))=\rk(GY\T)\ \text{ for all matrices $Y\in\F^{\cdotp\times n}$}.
\end{equation}
Clearly, $\cM_G$ depends only on the row space of~$G$; in other word, $\cM_{G}=\cM_{UG}$ for all $U\in\GL_k(\F_{q^m})$.
In the context of the rank-metric on $\F_{q^m}^n$ the row space of~$G$ is known as the
\emph{rank-metric code generated by~$G$}. Many invariants of the code (but not all) can be determined by the associated $q$-matroid.
For further details we refer to \cite{JuPe18,GJLR19,BCIJ21,GLJ22Gen} and Section~6 in the survey \cite{Gor21}.
We will briefly touch on the relation to rank-metric codes in \cref{D-Supp} and \cref{T-SuppCPerp}.
Furthermore, representability of the uniform \qM{}s can easily be described with the aid of maximum rank-distance codes (MRD codes).

\begin{exa}\label{E-UnifRepr}
Let $\cU_{k,n}(q)$ be the \emph{uniform \qM{} of rank~$k$} with ground space~$\F_q^n$; that is, the rank function is given by
$\rho(V)=\min\{k,\dim V\}$ for all $V\in\cL(\F_q^n)$.
Then $\cU_{0,n}(q)$ and $\cU_{n,n}(q)$ are representable over~$\F_q$. Precisely,  $\cU_{0,n}(q)$ is represented by the $1\times n$-zero matrix and
$\cU_{n,n}(q)$ by the $n\times n$-identity matrix.
We call $\cU_{0,n}(q)$ and $\cU_{n,n}(q)$ the \emph{trivial} and the \emph{free} \qM{} over~$\F_q^n$, respectively.
For $0<k<n$, the uniform \qM{} $\cU_{k,n}(q)$ is representable over $\F_{q^m}$ if and only if $m\geq n$.
Indeed, a matrix $G\in\F_{q^m}^{k\times n}$ represents $\cU_{k,n}(q)$ iff $\rk(GY\T)=k$ for all $Y\in\F_q^{k\times n}$ of rank~$k$.
But this is equivalent to~$G$ generating an MRD code \cite[Thm.~2 and~3]{Gab85}
and such a matrix~$G$ exists if and only if $m\geq n$; see \cite[Rem.~3.10]{Gor21}.
\end{exa}

Unsurprisingly, not every $q$-matroid is representable; see  \cite[Sec.~4]{GLJ22Gen} and   \cite[Sec.~3.3]{CeJu21}.
In the next section we will present examples showing that the direct sum of representable $q$-matroids need not be representable.

Before turning to the direct sum of $q$-matroids, we briefly summarize some properties of the direct sum of (classical) matroids.
We refer to Section~4.2 in the monograph \cite{Ox11} by Oxley for further details.

\begin{rem}\label{R-DirSumClassMatr}
Let $M_i=(S_i,r_i),i=1,2,$ be two matroids on disjoint ground sets~$S_1$ and~$S_2$.
Then $M=(S,r)$, where $S=S_1\cup S_2$ and $r(A)=r_1(A\cap S_1)+r_2(A\cap E_2)$ for all $A\subseteq S$, is a
matroid and called the direct sum of ~$M_1$ and~$M_2$, denoted by $M_1\oplus M_2$.
We denote by $\cI(M_1\oplus M_2)$ and~$\cI(M_i)$ the collections of independent sets in~$M_1\oplus M_2$ and~$M_i$, respectively, and use similar notation for the
collections of  circuits, flats, and open sets (the definitions of these sets are as in \cref{D-qMatroid} where we replace the dimension by the cardinality and the vector space sum by the set-theoretic union).
The following holds true.
\begin{alphalist}
\item The independent sets satisfy $\cI(M_1\oplus M_2)=\{I_1\cup I_2\mid I_i\in\cI(M_i)\}$.
\item The circuits satisfy $\cC(M_1\oplus M_2)=\cC(M_1)\cup\cC(M_2)$.
\item The flats satisfy $\cF(M_1\oplus M_2)=\{F_1\cup F_2\mid F_i\in\cF(M_i)\}$.
\item The open sets satisfy $\cO(M_1\oplus M_2)=\{O_1\cup O_2\mid O_i\in\cO(M_i)\}$.
\item If $M_i$ is represented by the matrices $G_i$ (over the same field~$\F$) for $i=1,2$, then $M_1\oplus M_2$ is represented by the matrix
\begin{equation}\label{e-GMat}
     \begin{pmatrix}G_1&0\\0&G_2\end{pmatrix}.
\end{equation}
\end{alphalist}
\end{rem}

The above implies that the direct sum of matroids can also be defined by way of independent sets, circuits, flats, or open sets (and many more).
It is easy to see that the obvious $q$-analogue of the rank function above or of any of the properties  (a)~--~(d) does not lead to a well-defined notion
for a direct sum of $q$-matroids.
In addition,~(e) does not generalize.
In fact, if~$G_i$ represents the \qM{}~$\cM_i$, then the $q$-matroid represented by~$G$ in \eqref{e-GMat} (in the sense of \cref{T-ReprqMatr})
depends on the choice of the representing matrices~$G_i$.
This is not surprising as the rank of
\[
   \begin{pmatrix}G_1&0\\0&G_2\end{pmatrix}\begin{pmatrix}Y_1\T\\Y_2\T\end{pmatrix}=\begin{pmatrix}G_1Y_1\T\\G_2Y_2\T\end{pmatrix}
\]
depends on the relation between~$G_1$ and~$G_2$.
This has already been observed in \cite[Rem.~22]{CeJu21} and  \cite[Prop.~4.5]{GLJ21C}.
The latter is used in \cite{GLJ21C}  to establish the non-existence of a coproduct in the category of $q$-matroids with strong maps as morphisms.

Fortunately, in \cite{CeJu21} the authors introduce a notion of direct sum of $q$-matroids.
We will present a slightly different definition of the same construction.
Its form will be beneficial for further computations.

Throughout the paper we will use the following notation.

\begin{nota}\label{Nota}
For any direct sum $E=E_1\oplus E_2$ of $\F$-vector spaces~$E_1$ and~$E_2$ we denote by
$\pi_i:E\longrightarrow E_i$  the corresponding projections.
Furthermore, for $n_1,n_2\in\N$ and $n=n_1+n_2$, we set $\F^n=\F^{n_1}\oplus\F^{n_2}$ such that
$\pi_1:\F^n\longrightarrow \F^{n_1}$  and $\pi_2:\F^n\longrightarrow \F^{n_2}$ are the projections onto the first~$n_1$ and
last~$n_2$ coordinates, respectively.
In order to accommodate for the correct matrix sizes, we write $\F^{n_1}\oplus 0$ and $0\oplus\F^{n_2}$ for
$\F^{n_1},\,\F^{n_2}$ considered as subspaces of~$\F^n$.
\end{nota}

\begin{theo}[\mbox{\cite[Sec.~7]{CeJu21}}]\label{T-DirSum}
Let $\cM_i=(E_i,\rho_i),\,i=1,2,$ be \qM{}s and set $E=E_1\oplus E_2$.
Define $\rho'_i:\cL(E)\longrightarrow \N_0,\ V\longmapsto \rho_i(\pi_i(V))$ for $i=1,2$.
Then $\cM'_i=(E,\rho'_i)$ is a \qM{}.
Set
$\cX=\{X\in\cL(E)\mid \rho_1'(X)+\rho_2'(X)<\dim X\}$ and $\cX_0=\cX\cup\{0\}$.
Define
\begin{equation}\label{e-rho}
  \rho:\cL(E)\longrightarrow\N_0,\quad V\longmapsto\dim V+\min_{X\in\cX_0\cap\cL(V)}\big(\rho'_1(X)+\rho'_2(X)-\dim X\big) \ \text{ for } V\in\cL(E).
\end{equation}
Then $\cM:=(E,\rho)$ is a $q$-matroid, called the \emph{direct sum of $\cM_1$ and~$\cM_2$} and denoted by $\cM_1\oplus\cM_2$.
A space~$V$ is dependent in~$\cM$ iff $\cX\cap\cL(V)\neq\emptyset$, and as a consequence the collection of circuits of~$\cM_1\oplus\cM_2$ is given by
$\cC(\cM_1\oplus\cM_2)=\{X\in\cX\mid X \text{ is inclusion-minimal in }\cX\}$.
\end{theo}

A short proof showing that~$\cM_i'$ is a \qM{} is given at \cite[Thm.~5.2]{GLJ21C} (see also \cite[Thm.~41]{CeJu21} for the case where $\dim E_j=1$ for $j\neq i$).
The fact that $\rho$ is a rank function can be found in \cite[Thm.~29]{CeJu21} and also follows
from \cite[Thm.~3.6 and its proof]{GLJ22Ind}.
In both references, $\rho(V)$ is defined as
\begin{equation}\label{e-rhoDSAlt}
    \rho(V)=\dim V+\min_{X\leq V}\big(\rho'_1(X)+\rho'_2(X)-\dim X\big),
\end{equation}
but since $\rho'_1(0)+\rho'_2(0)-\dim 0=0$, we can restrict ourselves to taking the minimum over $\cX_0\cap\cL(V)$.

The direct sum has some expected properties; see \cite[Thm.~47, Cor.~48]{CeJu21} or \cite[Thm.~5.6]{GLJ23DSCyc}.

\begin{theo}\label{T-DirSumProp}
Consider the situation from \cref{T-DirSum} and let $V_i\in\cL(E_i)$.
Then $\rho_i(V_i)=\rho'_i(V_i)$ and $\rho'_j(V_i)=0$ if $j\neq i$.
Furthermore, $\rho(V_1\oplus V_2)=\rho_1(V_1)+\rho_2(V_2)$.
As a consequence $\rho(\cM_1\oplus\cM_2)=\rho_1(\cM_1)+\rho_2(\cM_2)$.
\end{theo}

Note that by the above,~$E_i$ is a loop space of~$\cM'_j$ for $j\neq i$ (by definition a loop space is a space with rank~$0$).
The process from~$\cM_j$ to $\cM'_j$ in \cref{T-DirSum} is called \emph{adding a loop space} in \cite{CeJu21}.
It is a special instance of the direct sum because $\cM'_j=\cM_j\oplus\cU_0(E_i)$ for $i\neq j$, where $\cU_0(E_i)$ is the
trivial $q$-matroid on the ground space $E_i$ (that is, all spaces have rank value~$0$).
The process of adding a loop space can easily be described for representable $q$-matroids.

\begin{rem}\label{R-M1prime}
Let $n_1,n_2\in\N$.
Let $\cM_1=(\F^{n_1},\rho_1)$ be a representable $q$-matroid, say $\cM_1=\cM_G$ for some $G\in\F_{q^m}^{k\times n_1}$.
Define $\cM_1'=(\F^{n_1+n_2},\rho_1')$, where $\rho_1'(V)=\rho_1(\pi_1(V))$ for all $V\in\cL(\F^{n_1+n_2})$. Then
$\cM_1'=\cM_{G'}$, where $G'=(G\mid 0)\in\F^{k\times(n_1+n_2)}$.
\end{rem}

As already alluded to, the independent spaces, flats, etc.\ do not behave well under taking direct sums of \qM{}s.
Only one structure behaves nicely: the cyclic flats.
We briefly mention the following properties. They will not be needed for the rest of the paper, but rather serve to illustrate the contrast
to the direct sum of classical matroids; see \cref{R-DirSumClassMatr}(a)~--(d).
The discrepancy with respect to representability will be discussed in the next section.

\begin{rem}[\mbox{see \cite{GLJ23DSCyc}}]\label{R-DSPropSpaces}
Let $\cM_i=(E_i,\rho_i),\,i=1,2,$
and $\cM=\cM_1\oplus\cM_2$.
As in \cref{R-DirSumClassMatr} we use the notation $\cI(\cM)$ for the collection of independent spaces, and similarly for the flats, circuits, and open spaces. Furthermore, set $\cZ(\cM)=\cF(\cM)\cap\cO(\cM)$. Its elements are called the \emph{cyclic flats} of~$\cM$.
We set $\cI(\cM_1)\oplus\cI(\cM_2)=\{I_1\oplus I_2\mid I_i\in\cI(\cM_i)\}$ and similarly for the other collections.
Then
$\cI(\cM_1)\oplus\cI(\cM_2)\subset\cI(\cM),\ \;  \cF(\cM_1)\oplus\cF(\cM_2)\subset\cF(\cM),\ \; \cO(\cM_1)\oplus\cO(\cM_2)\subset\cO(\cM)$,
and $\cC(\cM_i)\subset\cC(\cM)$ for $i=1,2$. In general equality does not hold in any of these cases, but, on the positive side, we have
$\cZ(\cM_1)\oplus\cZ(\cM_2)=\cZ(\cM).$
\end{rem}

\section{Representability of the Direct Sum}

In this section we turn to representability of the direct sum of representable $q$-matroids.
We will provide an example of $\F_4$-representable $q$-matroids over the ground space $\F_2^4$ for which the direct sum is not
representable over any field extension $\F_{2^m}$.
A crucial ingredient will be paving $q$-matroids defined below.
We show that if $\cM_1,\,\cM_2$ are paving $q$-matroids of the same rank and represented by matrices~$G_i$ satisfying a certain condition,
then the direct sum $\cM_1\oplus\cM_2$ is not represented by the block diagonal matrix $\text{diag}(G_1,G_2)$.
This will be used to create the desired example.
A second example will be provided where the direct sum is representable only over larger fields than the summands.

Note that in the context of representability, it suffices to consider \qM{}s with ground space~$\F^n$.

\begin{defi}\label{D-Pav}
A $q$-matroid $\cM=(\F^n,\rho)$ is \emph{paving} if $\dim C\geq \rho(\cM)$ for all circuits~$C$ of~$\cM$.
\end{defi}

A large class of paving $q$-matroids is given as follows.

\begin{exa}\label{E-Pav}
Let $n\in\N$ and $1\leq k< n$. Let $\cS$ be a collection of $k$-dimensional subspaces of~$\F^n$ with the property that
$\dim(X\cap Y)\leq k-2$ for all distinct $X,Y\in \cS$. For $V\in\cL(\F^n)$ define
\[
    \rho(V)=\left\{\begin{array}{cl}k-1&\text{if }V\in\cS,\\[.5ex] \min\{k,\dim V\}&\text{otherwise.}\end{array}\right.
\]
Then $\cM=(\F^n,\rho)$ is  $q$-matroid of rank~$k$; see \cite[Prop.~4.6]{GLJ22Gen}.
The circuits are the spaces in~$\cS$ and all $(k+1)$-dimensional spaces not containing a space in~$\cS$.
Thus~$\cM$ is paving.
We will make use of this construction later in \cref{E-DSNonRepr1}.
Further examples  can be found in \cite[Sec.~4]{GLJ22Gen}.
\end{exa}

For the following result recall \cref{Nota}.

\begin{prop}\label{P-PavDS}
Let $\cM_i=(\F^{n_i},\rho_i),i=1,2,$ be paving $q$-matroids with $\rho_1(\cM_1)=\rho_2(\cM_2)=k$.
Let $n=n_1+n_2$ and $\cM=(\F^n,\rho)=\cM_1\oplus\cM_2$.
Denote by $\cC_1,\,\cC_2$ and $\cC$ the collections of circuits of~$\cM_1,\,\cM_2$, and~$\cM$, respectively.
Then every $C\in\cC$ has dimension at least~$k$ and
\[
   \{C\in\cC\mid \dim C=k\}
   =\{C_1\oplus 0\mid C_1\in\cC_1,\, \dim C_1=k\}\cup\{0\oplus C_2\mid C_2\in\cC_2,\, \dim C_2=k\}.
\]
\end{prop}

For the proof recall that in any \qM{} $(E,\rho)$ a $k$-dimensional circuit $C$ satisfies $\rho(C)=k-1$.

\begin{proof}
We start with the stated identity. Denote the set on the right hand side by~$\cV$.
\\
``$\supseteq$'' Let $C_1\oplus 0\in\cV$. From \cref{T-DirSumProp} we obtain $\rho(C_1\oplus0)=\rho_1(C_1)=k-1$, and
thus $C_1\oplus0$ is dependent in~$\cM$.
Clearly, every subspace of $C_1\oplus0$ is of the form $I_1\oplus0$ where $I_1$ is independent in $\cM_1$.
Using again \cref{T-DirSumProp} we conclude that $I_1\oplus0$ is independent in~$\cM$ and thus $C_1\oplus0$ is a circuit of~$\cM$ of dimension~$k$.
The same reasoning holds for $0\oplus C_2\in\cV$.
\\
``$\subseteq$''
By \cref{T-DirSum} a subspace $C\in\cL(\F^n)$ is a circuit of~$\cM$ if and only if it is inclusion-minimal subject to
\begin{equation}\label{e-CircCond}
   \rho_1(\pi_1(C))+\rho_2(\pi_2(C))\leq \dim C-1,
\end{equation}
where $\pi_1,\,\pi_2$ are the projections from $\F^n$ to the first $n_1$ and last $n_2$ coordinates, respectively.
Let $C\in\cC$ and dim $C\leq k$.
\\
1) Let $\dim\pi_1(C)=k$ (which implies $\dim C=k$).
Then \eqref{e-CircCond} implies that $\pi_1(C)$ is a dependent space of~$\cM_1$, and thus a circuit thanks to the paving property.
Hence $\rho_1(\pi_1(C))=k-1$ and thus $\rho_2(\pi_2(C))=0$  by \eqref{e-CircCond}.
But then $\pi_2(C)=0$ by the paving property of~$\cM_2$, and thus $C=\pi_1(C)\oplus0$.
In the same way we have $C=0\oplus\pi_2(C)$ if $\dim\pi_2(C)=k$.
\\
2) Let $\dim\pi_i(C)=\ell_i<k$ for $i=1,2$. Then $\pi_i(C)$ is independent in $\cM_i$ and $\rho_i(\pi_i(C))=\ell_i$.
Now we obtain $\dim C\leq\dim\pi_1(C)+\dim\pi_2(C)=\ell_1+\ell_2=\rho_1(\pi_1(C))+\rho_2(\pi_2(C))$ in contradiction to~\eqref{e-CircCond}.
Hence this case does not arise.
This also shows that all circuits have dimension at least~$k$, and the proof is complete.
\end{proof}

We now turn to representability of the direct sum and start with the following unsurprising result.

\begin{theo}\label{T-DSReprMat}
Let $\cM_i=(\F^{n_i},\rho_i),i=1,2,$ be $q$-matroids of rank~$k_i$.
Let $n=n_1+n_2$ and $\cM=\cM_1\oplus\cM_2$.
Suppose~$\cM$ is representable over~$\F_{q^m}$.
Then $\cM_1$ and~$\cM_2$ are representable over~$\F_{q^m}$ and $\cM=\cM_G$ for a matrix~$G$ of the form
\[
    G=\begin{pmatrix}G_1&0\\ 0&G_2\end{pmatrix},
\]
where $G_i\in\F_{q^m}^{k_i\times n_i}$ are such that $\cM_{G_i}=\cM_i$.
\end{theo}

\begin{proof}
Let $\cM=(\F^n,\rho)$ and suppose $\cM=\cM_G$ for some matrix~$G$ over $\F_{q^m}$.
Since $\rho(\cM)=k:=k_1+k_2$ by \cref{T-DirSumProp}, we may assume that~$G$ is in $\F_{q^m}^{k\times n}$ and has rank~$k$.
Furthermore, without loss of generality let~$G$ be in RREF.
Then we can write~$G$ as
\[
     G=\begin{pmatrix}G_1&G'&\\0&G_2\end{pmatrix}
\]
for some matrices $G_1\in\F_{q^m}^{t_1\times n_1},\,G_2\in\F_{q^m}^{t_2\times n_2}$ of full row rank  and
some matrix $G'\in\F_{q^m}^{t_1\times n_2}$. Hence $t_1+t_2=k_1+k_2$.
We show that $\cM_1=\cM_{G_1}$.
To do so, let $Y\in\F^{\cdotp\times n_1}$ and set $\hat{Y}=(Y\mid 0)$, which is in $\F^{\cdotp\times n}$.
Then $V:=\rowsp(Y)$ is in $\cL(\F^{n_1})$ and $V\oplus0=\rowsp(\hat{Y})$.
With the aid of \eqref{e-rkGY} and \cref{T-DirSumProp} we compute
\[
    \rk(G_1Y\T)= \rk (G\hat{Y}\T)=\rho(\rowsp(\hat{Y}))=\rho(V\oplus0)=\rho_1(V).
\]
This shows that $\cM_1=\cM_{G_1}$.
As a consequence, $t_1=k_1$ and $t_2=k_2$.
Next, $k_2=\rho_2(\F^{n_2})=\rho(0\oplus\F^{n_2})=\rho(\rowsp(0\mid I_{n_2}))=\rk\Smalltwomat{G'}{G_2}$.
Since $\rk G_2=k_2$ this implies
$\rowsp(G')\subseteq\rowsp(G_2)$.
Using that~$G$ is in RREF, we conclude $G'=0$ and hence~$G$ is block diagonal.
In the same way as above we obtain $\cM_2=\cM_{G_2}$.
This concludes the proof.
\end{proof}

We will now make use of some notions that are standard in the theory of rank-metric codes.
The proof of \cref{T-SuppCPerp} below illustrates the well known fact that for a representable
$q$-matroid $\cM_G$ the dimension of a dependent space equals the rank-weight of a suitable codeword in the dual code
$\rowsp(G)^\perp=\ker G$ (where the dual is defined with respect to the standard inner product).

\begin{defi}\label{D-Supp}
For a vector $v=(v_1,\ldots,v_n)\in\F_{q^m}^n$ we define the $\F$-\emph{support} of~$v$ as the subspace
\[
   \supp(v)=\subspace{v_1,\ldots,v_n}_{\F}\leq\F_{q^m}.
\]
The \emph{rank-weight} of~$v$ is $\wtrk(v):=\dim\supp(v)$.
\end{defi}

\begin{theo}\label{T-SuppCPerp}
For $i=1,2$ let $G_i\in\F_{q^m}^{k\times n_i}$ be of rank~$k$ and  $\cM_i=\cM_{G_i}$ be the associated $q$-matroids.
Suppose $\cM_1$ and $\cM_2$ are both paving.
Suppose furthermore that there exist vectors $v_i\in\ker G_i$ such that $\wtrk(v_1)=\wtrk(v_2)=k$
and $\supp(v_1)=\supp(v_2)$.
Then $\cM=\cM_1\oplus\cM_2$ is not represented by
\begin{equation}\label{e-Gdiag}
     G=\begin{pmatrix}G_1&0\\0&G_2\end{pmatrix}.
\end{equation}
\end{theo}

\begin{proof}
Let $\cN=\cM_G$, that is, $\cN$ is the $q$-matroid generated by~$G$.
We will show that~$\cN$ and~$\cM$ do not have the same circuits of dimension at most~$k$.
\\
First of all, \cref{P-PavDS} implies that all circuits of~$\cM$ have dimension at least~$k$, and the $k$-dimensional ones are also circuits of~$\cN$ thanks to their form described in that proposition.
\\
We will show the existence of a circuit of~$\cN$ of dimension at most~$k$ that is not a circuit of~$\cM$.
To do so, let $\supp(v_1)=\subspace{\alpha_1,\ldots,\alpha_k}$ for some ($\F$-linearly independent) $\alpha_i\in\F_{q^m}$.
Set $\alpha=(\alpha_1,\ldots,\alpha_k)\in\F_{q^m}^k$.
Then there exist matrices $Y_i\in\F^{k\times n_i}$ of rank~$k$ such that
$\alpha Y_i=v_i$ for $i=1,2$.
Hence $0=G_i Y_i\T\alpha\T$, and thus $\rk(G_iY_i\T)<k$ for $i=1,2$.
This shows that  $V_i:=\rowsp(Y_i)\in\cL(\F^{n_i})$ is a dependent space of~$\cM_i$.
Since $\dim V_i=k$, the paving property implies that~$V_i$ is a circuit of~$\cM_i$.
Now we have
\[
    G\begin{pmatrix}Y_1\T\\ Y_2\T\end{pmatrix}\alpha\T=\begin{pmatrix}0\\0\end{pmatrix},
\]
which means that $G(Y_1\mid Y_2)\T$ has rank less than~$k$.
Therefore $W:=\rowsp(Y_1\mid Y_2)\in\cL(\F^{n_1+n_2})$ is a $k$-dimensional dependent space of~$\cN$.
Since~$Y_1$ and~$Y_2$ are both nonzero, $W$ is not a circuit of the direct sum~$\cM$ thanks to \cref{P-PavDS}.
Since~$\cM$ does not have any circuits of dimension less than~$k$, we conclude that~$W$ contains a circuit of~$\cN$ of dimension at most~$k$
that is not a circuit of~$\cM$.
This implies that $\cN\neq\cM$, and~$G$ does not represent~$\cM$.
\end{proof}

Now we are ready to provide an example of a direct sum of representable $q$-matroids that is not representable over any field extension.

\begin{prop}\label{E-DSNonRepr1}
Let $\F=\F_2$ and $\F_4=\{0,1,\omega,\omega+1\}$. Consider the matrix
\[
     G_1=\begin{pmatrix}1&\omega&0&\omega+1\\0&0&1&\omega\end{pmatrix}\in\F_4^{2\times 4}
\]
and set $\cM_1:=\cM_{G_1}=(\F^4,\rho_1)$.
Then $\cM_1\oplus\cM_1$ is not representable over any field extension~$\F_{2^m}$.
\end{prop}

\begin{proof}
1) We first show that~$\cM_1$ is of the form as in \cref{E-Pav} and thus paving of rank~$2$.
Set $\cV=\{\rowsp (Y_1),\, \rowsp(Y_2),\, \rowsp(Y_3),\, \rowsp(Y_4),\, \rowsp(Y_5)\}$, where
\begin{equation}\label{e-cVMat}
   Y_1\!=\!\begin{pmatrix}1\!&\!0\!&\!0\!&\!0\\0\!&\!1\!&\!0\!&\!0\end{pmatrix},Y_2\!=\!\begin{pmatrix}1\!&\!0\!&\!1\!&\!1\\0\!&\!1\!&\!0\!&\!1\end{pmatrix},
      Y_3\!=\!\begin{pmatrix}1\!&\!0\!&\!0\!&\!1\\0\!&\!0\!&\!1\!&\!1\end{pmatrix},Y_4\!=\!\begin{pmatrix}0\!&\!1\!&\!1\!&\!0\\0\!&\!0\!&\!0\!&\!1\end{pmatrix},
      Y_5\!=\!\begin{pmatrix}1\!&\!1\!&\!0\!&\!1\\0\!&\!0\!&\!1\!&\!0\end{pmatrix}.
\end{equation}
One can verify, by hand or with for instance SageMath, that
\begin{equation}\label{e-rhoG}
     \rho_1(V)=\left\{\begin{array}{cl} 1,&\text{for }V\in\cV,\\ \min\{2,\dim V\},&\text{otherwise.}\end{array}\right.
\end{equation}
For instance, $\rho_1(V)=1$ for all $1$-dimensional subspaces simply reflects that the columns of~$G_1$ are linearly independent over~$\F_2$. Furthermore, $G_1Y_2\T=\Smallfourmat{\omega}{1}{\omega+1}{\omega}$, which has rank~$1$.
Similarly, the rank of all subspaces can easily be verified.
As a consequence,~$\cM_1$ is of the form as in \cref{E-Pav} and thus paving of rank~$2$.
Clearly, the matrix
\[
     \hat{G}_1=\begin{pmatrix}1&\omega+1&0&\omega\\0&0&1&\omega+1\end{pmatrix},
\]
obtained from~$G$ by replacing the primitive element~$\omega$ by its conjugate~$\omega+1$,
also represents~$\cM_1$.
\\
2) We show that $G_1$ and~$\hat{G}_1$ are the only matrices over any field extension $\F_{2^m},\,m\geq1,$ that represent~$\cM_1$.
To do so, let $m\geq1$ and $H=(h_1,\,h_2,\,h_3,\,h_4)\in\F_{2^m}^{2\times 4}$ be such that $\cM_H=\cM_1$.
Thus $\rk(HY\T)=\rk(G_1Y\T)$ for all matrices $Y\in\F^{\cdotp\times4}$.
Without loss of generality we may assume that $H$ is in RREF.
Clearly $\rk(H)=2$.
Moreover, $h_1,\ldots,h_4$ are linearly independent over~$\F_2$ because~$\cM_1$ has no loops (a loop is a 1-dimensional space of rank 0).
Next, $\rk(HY_1\T)=1=\rk(HY_4\T)$ shows that
$h_2\in\subspace{h_1}_{\F_{2^m}}$ and $h_4\in\subspace{h_2+h_3}_{\F_{2^m}}$.
Hence $h_2$ and $h_4$ are not pivot columns of~$H$. All of this implies that~$H$ must be of the form
\[
   H=\begin{pmatrix}1&\alpha&0&\alpha\beta\\0&0&1&\beta\end{pmatrix}\text{ for some }\alpha,\,\beta\in\F_{2^m}.
\]
The $\F_2$-linear independence of the columns of~$H$ implies that $\alpha,\,\beta\not\in\F_2$.
Next,
\[
    1=\rk(HY_3\T)=\rk\begin{pmatrix}1+\alpha\beta&\alpha\beta\\ \beta&1+\beta\end{pmatrix}
       =\rk\begin{pmatrix}1&\alpha\beta\\1&1+\beta\end{pmatrix},
\]
and this results in $\alpha\beta=1+\beta$. Using this, we continue with
\[
  1=\rk(HY_2\T)=\rk\begin{pmatrix}1+\alpha\beta&\alpha+\alpha\beta\\1+\beta&\beta\end{pmatrix}
    =\rk\begin{pmatrix}\beta&\alpha+1+\beta\\1+\beta&\beta\end{pmatrix}
    =\rk\begin{pmatrix}\beta&\alpha+1\\1+\beta&1\end{pmatrix}.
\]
Using the determinant and $\alpha\beta=1+\beta$ we conclude that $\alpha+\beta=0$, hence $\alpha=\beta$, and with $\alpha\beta=1+\beta$ we arrive at
$\beta^2+\beta+1=0$. Thus $\beta\in\F_4\setminus\F_2$ (and~$m$ is even), and the two choices $\beta\in\{\omega,\omega+1\}$ lead to $H\in\{G_1,\hat{G}_1\}$.
\\
3) Consider now $\cM=\cM_1\oplus\cM_1$.
From 2) and \cref{T-DSReprMat} we know that if $\cM$ is $\F_{2^m}$-representable for some $m\in\N$, then~$m$ is even and any representing matrix is of the form
\[
  G=\begin{pmatrix}G'_1&0\\0&G'_2\end{pmatrix},\ \text{ where }G'_i\in\{G_1,\,\hat{G}_1\}.
\]
Since $v=(1,1,\omega,1)\in\ker G_1,\,\hat{v}=(1,1,\omega+1,1)\in\ker \hat{G}_1$ satisfy $\supp(v)=\supp(\hat{v})=\subspace{1,\omega}$ we may apply
\cref{T-SuppCPerp} and conclude that~$\cM$ is not representable over any field extension $\F_{2^m}$.
\end{proof}

In the following case, the direct sum is not representable over the same field as the summands, but over a field extension.
Recall the uniform \qM{}s from \cref{E-UnifRepr}.

\begin{prop}\label{E-DSNonRepr3}
Let $\cM_1=\cU_{1,2}(q)$.
Then~$\cM_1$ is representable over~$\F_{q^2}$, whereas~$\cM_1\oplus\cM_1$ is representable over $\F_{q^m}$ iff $m\geq4$.
\end{prop}

\begin{proof}
We know already from \cref{E-UnifRepr} that~$\cM_1$ is representable over $\F_{q^2}$.
Even more, for any $m\geq 2$ any matrix $G=\big(1\ \alpha\big)$ with $\alpha\in\F_{q^m}\setminus\F_q$ represents~$\cM_1$.
Consider now $\cM=\cM_1\oplus\cM_1=(\F^4,\rho)$.
With the aid of the very definition of the direct sum one easily verifies that $\rho(V)=1$ for all $1$-dimensional spaces $V\leq\F^4$, while
for the $2$-dimensional spaces
\begin{equation}\label{e-rhotilde2}
   \rho(V)=\left\{\begin{array}{cl} 1,&\text{if }V=\subspace{e_1,e_2}\text{ or }V=\subspace{e_3,e_4},\\ 2,&\text{otherwise,}\end{array}\right.
\end{equation}
where $e_1,\ldots,e_4$ are the standard basis vectors of~$\F^4$.
Suppose~$\cM$ is representable over $\F_{q^m}$.
\cref{T-DSReprMat} tells us that~$\cM$ has a representing matrix of the form
\[
  G=\begin{pmatrix}1&\beta&0&0\\0&0&1&\gamma\end{pmatrix} \ \text{ for some }\beta,\gamma\in\F_{q^m}\setminus\F.
\]
1) Let $m=2$. Then $\F_{q^2}=\subspace{1,\beta}$. Thus $\gamma=a+b\beta$ for some $a,b\in\F$.
Set
\[
     V=\rowsp(Y), \text{ where }Y=\begin{pmatrix}a&b&0&1\\1&0&1&0\end{pmatrix}.
\]
Then $\dim V=2$ and $\rk(GY\T)=1$. This contradicts \eqref{e-rhotilde2}, and thus~$G$ does not represent~$\cM$.
\\
2) Let $m=3$. Then $\F_{q^3}=\subspace{1,\beta,\beta^2}$ and
\[
  \gamma=c_0+c_1\beta+c_2\beta^2,\quad \beta^3=b_0+b_1\beta+b_2\beta^2\text{ for some }b_i,\,c_i\in\F.
\]
This implies $\beta\gamma=c_2b_0+(c_0+c_2b_1)\beta+(c_1+c_2b_2)\beta^2$.
Now we have the following cases:
\\
i) If $c_2=0$, then $\rk(GY\T)=1$ for
\[
    Y=\begin{pmatrix}c_0&c_1&0&1\\1&0&1&0\end{pmatrix}.
\]
ii) If $c_2\neq0=c_1+c_2b_2$, then $\rk(GY\T)=1$ for
\[
    Y=\begin{pmatrix}c_2b_0&c_0+c_2b_1&0&1\\0&1&1&0\end{pmatrix}.
\]
iii) If $c_2\neq0\neq c_1+c_2b_2$, we have $(c_1+c_2b_2)\gamma-c_2\beta\gamma=f_0+f_1\beta$ for some $f_i\in\F$ and thus
$\rk(GY\T)=1$ for
\[
    Y=\begin{pmatrix}f_0&f_1&0&1\\c_1+c_2b_2&-c_2&1&0\end{pmatrix}.
\]
In all cases we obtain a contradiction to \eqref{e-rhotilde2} and conclude that
$\cM$ is not representable over $\F_{q^3}$.
\\
3) Let $m\geq 4$. We show that~$\cM$ is represented by the matrix
\[
    G=\begin{pmatrix}1&z&0&0\\0&0&1&z^2\end{pmatrix} \ \text{ for any $z\in\F_{q^m}$ of degree at least~$4$}.
\]
Denote by $\cD(\cM)$ and $\cD(\cM_G)$ the collections of dependent spaces of~$\cM$ and~$\cM_G$, respectively.
We will show that
\begin{equation}\label{e-DepSpacestildeG}
    \cD(\cM_{G})=\cD(\cM).
\end{equation}
To do so, we first determine the set $\cX=\{X\in\cL(\F^4)\mid \rho_1(\pi_1(X))+\rho_1(\pi_2(X))<\dim X\}$ from  \cref{T-DirSum}.
Since $\cM_1$ is the uniform $q$-matroid of rank~$1$, we have
$\rho_1(\pi_1(X))+\rho_1(\pi_2(X))=\min\{1,\dim\pi_1(X)\}+\min\{1,\dim\pi_2(X)\}\leq 2$ for all~$X\in\cL(\F^4)$,
which together with \eqref{e-rhotilde2} implies
\[
    \cX=\{X\leq\F^4\mid \dim X\geq 3\}\cup\{\subspace{e_1,e_2},\subspace{e_3,e_4}\}.
\]
The form of~$\cX$ shows that every subspace of $\F^4$ containing a space in~$\cX$ is itself in~$\cX$.
Now it follows from \cref{T-DirSum} that
\[
      \cX=\cD(\cM).
\]
We now turn to the $q$-matroid $\cM_{G}=(\F^4,\hat{\rho})$.
By definition $\hat{\rho}(\rowsp(Y))=\rk(GY\T)$ for any matrix $Y\in\F^{\cdotp\times 4}$.
Set $G'_1=(1\ \; z)$ and $G'_2=(1\ \; z^2)$, thus
\[
  G=\begin{pmatrix}G'_1&0\\0&G'_2\end{pmatrix}.
\]
Since both~$G'_1$ and $G'_2$ represent the $q$-matroid~$\cM_1$, any matrix $Y=(Y_1\mid Y_2)\in\F^{\cdotp\times 4}$ satisfies
\[
   \rk(G'_i Y_i\T)=\rho_1(\pi_i(\rowsp(Y))).
\]
Let now $Y=(Y_1\mid Y_2)\in\F^{\cdotp\times 4}$ be such that $\rowsp(Y)\in\cD(\cM)=\cX$. Then
\[
   \hat{\rho}(\rowsp(Y))=\rk(G Y\T)\leq \rk(G'_1Y_1\T)+\rk(G'_2Y_2\T)=\rho_1(\pi_1(\rowsp(Y)))+\rho_1(\pi_2(\rowsp(Y)))<\rk Y.
 \]
This shows that $\cD(\cM)\subseteq \cD(\cM_{G})$.
This can also be concluded from the coproduct property of $\cM\oplus\cM$; see \cite[Thm.~5.5]{GLJ21C}.
It remains to show the converse containment.
Let now $Y\in\F^{t\times 4}$ be a matrix of rank~$t$ such that $\rowsp(Y)\in \cD(\cM_G)$.
Thus
\begin{equation}\label{e-GtY}
    \rk(G Y\T)<t.
\end{equation}
We will show that $\rowsp(Y)$ is in~$\cX$.
For $t\geq 3$ every~$Y$ satisfies \eqref{e-GtY} and~$\rowsp(Y)$ is contained in~$\cX$, while for $t=1$ no matrix~$Y$ satisfies \eqref{e-GtY}.
It remains to consider the case $t=2$.
Let
\[
   Y=\begin{pmatrix}a_1&b_1&c_1&d_1\\a_2&b_2&c_2&d_2\end{pmatrix}
\]
be of rank~$2$ and satisfying \eqref{e-GtY}. Without loss of generality let~$Y$ be in RREF.
We compute
\[
   GY\T=\begin{pmatrix}a_1+b_1z&a_2+b_2z\\c_1+d_1z^2&c_2+d_2z^2\end{pmatrix}
\]
and
\[
   \det(GY\T)=(a_1c_2-a_2c_1)+(b_1c_2-b_2c_1)z+(a_1d_2-a_2d_1)z^2+(b_1d_2-b_2d_1)z^3.
\]
Recall that $1,z,z^2,z^3$ are linearly independent over~$\F_q$ by choice of~$z$.
Thanks to the RREF of~$Y$ we only have to consider the following cases.
\\
i) If $a_1=b_1=a_2=b_2=0$, then  $\rowsp(Y)=\subspace{e_3,e_4}$ is in~$\cX$.
\\
ii) Let $a_1=1,a_2=0,b_1=0,b_2=1$. Then  $\det(GY\T)=c_2-c_1z+d_2z^2-d_1z^3$  and \eqref{e-GtY} implies $c_1=c_2=d_1=d_2=0$.
Thus $\rowsp(Y)=\subspace{e_1,e_2}$, which is in $\cX$.
\\
iii) Let $a_1=1$ and $a_2=b_2=0$. Then  $\det(G Y\T)=c_2+b_1c_2z+d_2z^2+b_1d_2z^3$ and \eqref{e-GtY} implies $c_2=d_2=0$.
But this contradicts that~$Y$ has rank~$2$ and thus this case does not occur.
\\
iv) Let $a_1=a_2=0=b_2$ and $b_1=1$. Then  $\det(G Y\T)=c_2z+d_2z^3$ and  \eqref{e-GtY} implies  $c_2=d_2=0$, which contradicts $\rk Y=2$.
Hence this case does not occur either.
\\
All of this establishes~\eqref{e-DepSpacestildeG}.
Thanks to \cref{R-Crypto}, more specifically \cite[Cor.~73]{BCJ22}, we arrive at $\cM=\cM_G$, and this shows the $\F_{q^m}$-representability of~$\cM$.
\end{proof}

Thus far we provided two examples. In the first one, the direct sum of representable \qM{}s is not representable, and in the second one it requires
a larger field size for a representation than the summands.
In fact, almost all examples that we were able to compute fall into one of the two categories.
The only examples known to us where the direct sum is representable over the same field as both summands, are such that one
summand is representable over~$\F_q$.
While we are not able to prove a general statement about these specific direct sums, we can discuss two extreme cases.
If $\cM_1$ is the trivial \qM{} $\cU_{0,n_1}(q)$ we have seen already in \cref{R-M1prime} and the paragraph preceding it, that
$\cM_1\oplus\cM_2$ is representable over the same field as~$\cM_2$.
The other extreme, where~$\cM_1$ is the free \qM{}, is dealt with in the following proposition.

\begin{prop}
Let $\cM_1=\cU_{n_1,n_1}(q)$, which is representable by $I_{n_1}$, and let the \qM{} $\cM_2=(\F_q^{n_2},\rho_2)$ be represented by $G_2\in\F_{q^{m}}^{k_2\times n_2}$.
Then $\cM_1\oplus\cM_2$ is representable over~$\F_{q^{m}}$ and a representing matrix is given by
\[
     G= \begin{pmatrix}I_{n_1}&0\\0&G_2\end{pmatrix}.
\]
\end{prop}

\begin{proof}
Set $n=n_1+n_2$. We first consider the \qM{} $\cM_G$ and determine its rank function, which we denote by~$\rho_G$.
Let $V\in\cL(\F^n)$ and $V=\rowsp(Y)$ where $Y\in\F^{\ell\times n}$ of rank~$\ell$.
Without loss of generality, we may assume that~$Y$ is in reduced row echelon.
Then~$Y$ partitions as
\begin{equation}\label{e-Y}
   Y=\begin{pmatrix}Y_1&\hat{Y}\\0&Y_2\end{pmatrix},
\end{equation}
for some $Y_i\in\F^{\ell_i\times n_i}$ and $\hat{Y}$ of corresponding size, and where $\rk(Y_i)=\ell_i$ for $i=1,2$.
This means that $\rowsp(Y_1)=\pi_1(V)$, which has dimension~$\ell_1$, and
$\rowsp(Y_2)=\pi_2(V\cap (0\oplus\F^{n_2}))$, which has dimension~$\ell_2$.
Now we have
\[
   GY\T=\begin{pmatrix}Y_1\T&0\\G_2\hat{Y}\T& G_2 Y_2\T\end{pmatrix},
\]
and since $Y_1\T$ has full column rank, we obtain the rank value
\begin{equation}\label{e-rhoDSG}
  \rho_G(V)=\rk(G Y\T)=\rk(V_1)+\rk(G_2 Y_2\T)=\dim\pi_1(V)+\rho_2(\pi_2(V\cap (0\oplus\F^{n_2}))).
\end{equation}
We now turn to  $\cM_1\oplus\cM_2$ and denote its rank function by~$\rho$.
We use \eqref{e-rhoDSAlt} to evaluate~$\rho$, thus
\begin{equation}\label{e-rhoDS}
    \rho(V)=\dim V+\min_{X\leq V}\sigma(X), \ \text{ where }\sigma(X)= \dim \pi_1(X)+\rho_2(\pi_2(X))-\dim X.
\end{equation}
Consider again $V=\rowsp(Y)$ with~$Y$ as  in\eqref{e-Y}.
Let $X$ be a subspace of~$V$ of dimension~$x$. Then~$X$ is of the form $X=\rowsp(SY)$ for some $S\in\F^{x\times\ell}$ of rank~$x$.
Again, we may assume~$S$ in reduced row echelon form and can partition the matrix as
\[
    S=\begin{pmatrix}S_1&\hat{S}\\0&S_2\end{pmatrix},
\]
where $S_i\in\F^{x_i\times\ell_i}$ of rank~$x_i$ and $\hat{S}$ accordingly.
Then
\[
   SY=\begin{pmatrix}S_1 Y_1&S_1\hat{Y}+\hat{S}Y_2\\ 0&S_2 Y_2\end{pmatrix}.
\]
The two diagonal blocks have full row rank and $\pi_1(X)=\rowsp(S_1Y_1)$ while $\pi_2(X)=\rowsp\Smalltwomat{S_1\hat{Y}+\hat{S}Y_2}{S_2 Y_2}$.
Now we have
\[
  \sigma(X)=\rk(S_1 Y_1)+\rho_2(\pi_2(X))-\dim X=\rho_2(\pi_2(X))-\rk(S_2 Y_2),
\]
which we want to minimize over all subspaces~$X$ of~$V$.
Note that this expression does not depend on~$\pi_1(X)$.
Furthermore, since $\rk(S_2Y_2)$ does not depend on~$\hat{S}$ and~$\rho_2$ is non-decreasing, the minimum is attained by a space $X$ of the form $X=\rowsp(0\mid S_2 Y_2)$.
Thus we arrive at
\[
    \sigma(X)=\rho_2(\pi_2(X))-\dim(\pi_2(X)),
\]
which we have to minimize over all subspaces $\pi_2(X)$ of $\rowsp(Y_2)$.
It is easy to see that the rank function~$\rho$ of any \qM{}  satisfies $\rho(V)-\dim V\leq\rho(W)-\dim W$ whenever $W\leq V$
(see for instance \cite[Lem.~2]{CeJu21}), and thus we conclude that the minimum is attained by a subspace~$X$ such that $\pi_2(X)=\rowsp(Y_2)$.
Choosing $X=\rowsp(0\mid Y_2)$, we may rewrite \eqref{e-rhoDS} as
\[
   \rho(V)=\dim V+\rho_2(\rowsp(Y_2))-\dim\rowsp(Y_2)=\rk(Y_1)+\rho_2(\rowsp(Y_2)),
\]
which agrees with \eqref{e-rhoDSG}.
All of this shows that $\cM_1\oplus\cM_2=\cM_G$, as stated.
\end{proof}

The following summarizes the open problems concerning representability of the direct sum.

\begin{OpenP}
Let $\cM_i=(\F_q^{n_i},\rho_i),\, i=1,2$, be representable \qM{}s. Let $\F_{q^{\widehat{m}}}$ be the smallest field over which both~$\cM_1,\,\cM_2$ are  representable.
\begin{alphalist}
\item Can one characterize representability of $\cM_1\oplus\cM_2$ in terms of the summands?
\item If $\cM_1\oplus\cM_2$ is representable, what is the smallest degree~$m$ such that  $\cM_1\oplus\cM_2$ is representable over~$\F_{q^m}$?
         Under which conditions do we have $m=\widehat{m}$?
\end{alphalist}
\end{OpenP}

We believe that studying representability of \qM{}s, especially finding obstructions to representability, will be a challenging, but highly instructive topic for a
better understanding of \qM{}s and their differences to matroids.
A characterization of representability of the direct sum in terms of the summands may be a first step in this direction.
Note that \cref{T-SuppCPerp} provides already an obstruction to representability in a special case.

\bibliographystyle{abbrv}

\end{document}